\documentclass[reqno,12pt]{amsart}
\usepackage{amsmath,epsfig}
\usepackage{amssymb}

\usepackage{graphicx}
\usepackage{xcolor}
\usepackage{color}

\usepackage{ulem}
\newtheoremstyle{rem}{1.3ex}{1.3ex}{\rmfamily}{}
{\itshape\rmfamily}{}{1.5ex}{}

\newtheorem{theorem}{Theorem}[section]

\newtheorem{corollary}[theorem] {Corollary}

\theoremstyle{definition}

\newtheorem{remark}[theorem] {Remark}

\renewcommand{\section}{\secdef\sct\sect}
\newcommand{\sct}[2][default]{\refstepcounter{section}
\setcounter{equation}{0}
\vspace{0.5cm}
\centerline{ \large
\scshape \arabic{section}.\ #1}
\vspace{0.3cm}}
\newcommand{\sect}[1]{
\vspace{0.5cm}
\centerline{\large\scshape #1}
\vspace{0.3cm}}

\renewcommand{\subsection}{\secdef \subsct\sbsect}
\newcommand{\subsct}[2][default]{\refstepcounter{subsection}
\nopagebreak
\vspace{0.5\baselineskip}
{\flushleft\bf \arabic{section}.\arabic{subsection}~\bf #1  }
\nopagebreak}
\newcommand{\sbsect}[1]{\vspace{0.1cm}\noindent
{\bf #1}\vspace{0.1cm}}

\def\phi{\varphi }

\newcommand{\C}     {\mathbb{C}}

\newcommand{\R}     {\mathbb{R}}

\newcommand{\E}     {\mathbb{E}}

\newcommand{\V}     {\mathbb{V}}

\def\1{{\mathchoice {1\mskip-4mu\mathrm l}
                    {1\mskip-4mu\mathrm l}
                    {1\mskip-4.5mu\mathrm l} {1\mskip-5mu\mathrm l}}}

\begin{document}

\title[Moderate deviations for Wigner matrices]{\large
Moderate deviations for\\\vspace{2mm}the eigenvalue counting function\\\vspace{5mm}of Wigner matrices}

\author[Hanna D\"oring and Peter Eichelsbacher]{} 
\maketitle
\thispagestyle{empty}
\vspace{0.2cm}

\centerline{\sc Hanna D\"oring\footnote{Ruhr-Universit\"at Bochum, Fakult\"at f\"ur Mathematik,
NA 3/68, D-44780 Bochum, Germany, {\tt hanna.doering@rub.de}}, Peter Eichelsbacher\footnote{Ruhr-Universit\"at Bochum, Fakult\"at f\"ur Mathematik,
NA 3/66, D-44780 Bochum, Germany, {\tt peter.eichelsbacher@rub.de}
\\Both authors have been supported by Deutsche Forschungsgemeinschaft via SFB/TR 12. The first author was supported by the international
research training group 1339 of the DFG.}}


\vspace{2 cm}

\begin{quote}
{\small {\bf Abstract:} }
We establish a moderate deviation principle (MDP) for the number of eigenvalues of a Wigner matrix in an interval.
The proof relies on fine asymptotics of the variance of the eigenvalue counting function of GUE matrices due to Gustavsson.
The extension to certain families of Wigner matrices is based on the Tao and Vu Four Moment Theorem and applies
localization results by Erd\"os, Yau and Yin. Moreover we investigate families of covariance matrices as well.
\end{quote}




\newpage
\setcounter{section}{0}

\section{Introduction}
Recently, in \cite{Dallaporta/Vu:2011} the Central Limit Theorem (CLT) for the eigenvalue counting function of Wigner matrices, that is the number
of eigenvalues falling in an interval, was established. This {\it universality result} relies on fine asymptotics of the variance of the eigenvalue counting function,
on the Fourth Moment Theorem due to Tao and Vu as well as on recent localization results due to Erd\"os, Yau and Yin. See also \cite{Dalla:2011}.
Our paper is concerned with the moderate deviation principle (MDP) of the eigenvalue counting function. We will start with the MDP
for Wigner matrices where the entries are Gaussian (the so-called Gaussian unitary ensemble (GUE)), first proven by the authors in \cite{DoeringEichelsbacher:2010}.  
Next we establish a MDP for individual eigenvalues
in the bulk of the semicircle law (which is an MDP corresponding to the Gaussian behaviour proved in \cite{Gustavsson:2005}). This MDP will be extended
to certain families of Wigner matrices by means of the Four Moment Theorem (see \cite{Tao/Vu:2009}, \cite{Tao/Vu:2010}). It seem to be the first application
of the Four Moment Theorem to be able to obtain not only universality of convergence in distribution but also to obtain deviations results on a logarithmic scale, universally.
Finally a strategy based on the MDP for individual eigenvalues in the bulk will be shown to imply the MDP for the eigenvalue counting function, 
universally for certain Wigner matrices. In the meantime, we successfully apply the Four Moment Theorem to obtain MDPs also at the edge
of the spectrum as well as for the determinant of certain Wigner matrices, see \cite{DoeringEichelsbacher:2012a}, \cite{DoeringEichelsbacher:2012b}.

Consider two independent families of i.i.d. random variables $(Z_{i,j})_{1 \leq i <j}$ (complex-valued) and $(Y_i)_{1 \leq i}$ (real-valued), zero mean, such that
$\E Z_{1,2}^2=0, \E|Z_{1,2}|^2=1$ and $\E Y_1^2=1$. Consider the (Hermitian) $n \times n$ matrix $M_n$ with entries $M_n^*(j,i) = M_n(i,j) = Z_{i,j} / \sqrt{n}
$ for $i <j $ and  $M_n^*(i,i) = M_n(i,i)=Y_i/ \sqrt{n}$. Such a matrix is called {\it Hermitian Wigner matrix}. An important example of Wigner matrices is the case
where the entries are Gaussian, giving rise to the so-called Gaussian Unitary Ensembles (GUE). GUE matrices will be denoted by $M_n'$. In this case, the joint
law of the eigenvalues is known, allowing a good description of their limiting behaviour both in the global and local regimes (see \cite{Zeitounibook}). In the
Gaussian case, the distribution of the matrix is invariant by the action of the group $SU(n)$. The eigenvalues of the matrix $M_n'$ are independent
of the eigenvectors which are Haar distributed. If  $(Z_{i,j})_{1 \leq i <j}$ are real-valued the {\it symmetric Wigner matrix} is defined analogously and
the case of Gaussian variables with $\E Y_1^2=2$ is of particular importance, since their law is invariant under the action of the orthogonal group $SO(n)$, known
as Gaussian Orthogonal Ensembles (GOE).

Denote by $\lambda_1, \ldots, \lambda_n$ the real eigenvalues of the normalised Hermitian (or symmetric) Wigner matrix $W_n = \frac{1}{\sqrt{n}} M_n$. 
The Wigner theorem states that the empirical measure
$$
L_n := \frac 1n \sum_{i=1}^n \delta_{\lambda_i}
$$
on the eigenvalues of $W_n$ converges weakly almost surely as $n \to \infty$ to the semicircle law
$$
d \varrho_{sc}(x) = \frac{1}{2 \pi} \sqrt{ 4 - x^2} \, 1_{[-2,2]}(x) \, dx,
$$ 
(see \cite[Theorem 2.1.21, Theorem 2.2.1]{Zeitounibook}). Consequently, for any interval $I \subset \R$,
$$
\frac{1}{n} N_I(W_n) := \frac 1n \sum_{i=1}^n 1_{\{\lambda_i \in I\}} \to \varrho_{sc}(I)
$$
almost surely as $n \to \infty$. At the fluctuation level, it is well known that for the GUE, $W_n' := \frac{1}{\sqrt{n}} M_n'$
satisfies a CLT (see \cite{Soshnikov:2000}): Let $I_n$ be an interval in $\R$. If $\V(N_{I_n}(W_n')) \to \infty$
as $n \to \infty$, then
$$
\frac{N_{I_n}(W_n') - \E [ N_{I_n}(W_n')]}{\sqrt{\V (N_{I_n}(W_n'))}} \to N(0,1)
$$
as $n \to \infty$ in distribution.

In \cite{Gustavsson:2005} the asymptotic behavior of the expectation and the variance of the counting function $N_{I_n}(W_n')$ for intervals $I_n=[y(n), \infty)$ with $y(n) = G^{-1}(k/n)$ (where $k=k(n)$ is such that $k/n \to a \in (0,1)$ -- \textit{strictly in the bulk}--, and $G$ denotes the distribution function 
of the semicircle law) was established:
\begin{equation} \label{asymp}
\E [ N_{I_n}(W_n')] = n -k(n)  + O \bigl( \frac{\log n}{n} \bigr) \,\, \text{and} \,\, \V (N_{I_n}(W_n')) = \bigl( \frac{1}{2 \pi^2} + o(1) \bigr) \, \log n.
\end{equation}
The proof applied strong asymptotics for orthogonal polynomials with respect to exponential weights, see \cite{Deift/Thomas:1999}.
In particular the CLT holds for $N_I(W_n')$ if $I=[y, \infty)$ with $y \in (-2,2)$, and moreover in this case one obtains
$$
\frac{N_{I}(W_n') - n \varrho_{sc}(I)}{\sqrt{\frac{1}{2 \pi^2} \log n}} \to N(0,1)
$$
as $n \to \infty$ (called the CLT with numerics). These conclusions were extended to non-Gaussian Wigner matrices in \cite{Dallaporta/Vu:2011}.

Certain deviations results and concentration properties for Wigner matrices were considered. Our aim is to establish certain moderate
deviation principles. Recall that a sequence of laws $(P_n)_{n \geq 0}$ on a Polish space $\Sigma$ satisfies a large deviation principle (LDP)
with good rate function $I : \Sigma \to \R_+$ and speed $s_n$ going to infinity with $n$ if and only if the level sets $\{x: I(x) \leq M\}$, $0 \leq M < \infty$,
of $I$ are compact and for all closed sets $F$
$$
\limsup_{n \to \infty} s_n^{-1} \log P_n(F) \leq - \inf_{x \in F} I(x)
$$
whereas for all open sets $O$
$$
\liminf_{n \to \infty} s_n^{-1} \log P_n(O) \geq - \inf_{x \in O} I(x).
$$
We say that a sequence of random variables satisfies the LDP when the sequence of measures induced by these variables satisfies the LDP. Formally
a moderate deviation principle is nothing else but the LDP. However, we speak about a moderate deviation principle (MDP) for a sequence of random variables,
whenever the scaling of the corresponding random variables is between that of an ordinary Law of Large Numbers (LLN) and that of a CLT.


Large deviation results for the empirical measures of Wigner matrices are still only known for the Gaussian ensembles since their
proof is based on the explicit joint law of the eigenvalues, see \cite{BenArous/Guionnet:1997} and \cite{Zeitounibook}. A moderate deviation
principle for the empirical measure of the GUE or GOE is also known, see  \cite{Dembo/Guionnet/Zeitouni:2003}. This moderate deviations result
does not have yet a fully universal version for Wigner matrices. It has been generalised to Gaussian divisible matrices with a deterministic self-adjoint matrix added
with converging empirical measure \cite{Dembo/Guionnet/Zeitouni:2003} and to Bernoulli matrices \cite{DoeringEichelsbacher:2009}.

Our first result is a MDP for the number of eigenvalues of a GUE matrix in an interval. It is a little modification 
of \cite[Theorem 5.2]{DoeringEichelsbacher:2010}. In the following, for two sequences of real numbers $(a_n)_n$ and $(b_n)_n$ we denote
by $a_n \ll b_n$ the convergence $\lim_{n \to \infty} a_n/b_n = 0$.

\begin{theorem} \label{GUE-MDP1} 
Let $M_n'$ be a GUE matrix and $W_n' := \frac{1}{\sqrt n} M_n'$. Let $I_n$ be an interval in $\R$. If $\V(N_{I_n}(W_n')) \to \infty$
for $n \to \infty$, then for any sequence $(a_n)_n$ of real numbers such that
\begin{equation} \label{regime}
1 \ll a_n \ll \sqrt{\V (N_{I_n}(W_n'))}
\end{equation}
the sequence $(Z_n)_n$ with
$$
Z_n = \frac{N_{I_n}(W_n') - \E [ N_{I_n}(W_n')]}{a_n \, \sqrt{\V (N_{I_n}(W_n'))}}
$$
satisfies a MDP with speed $a_n^2$ and rate function $I(x)=\frac{x^2}{2}$.
\end{theorem}

\begin{remark} \label{renew}
Let $I=[y, \infty)$ with $y \in (-2,2)$. An easy consequence of Theorem \ref{GUE-MDP1} is that the sequence $(\hat{Z}_n)_n$ with
$$
\hat{Z}_n = \frac{N_{I}(W_n') - n \varrho_{sc}(I)}{a_n \, \sqrt{\frac{1}{2 \pi^2} \log n}}
$$
satisfies the MDP with the same speed, the same rate function, and in the same regime \eqref{regime} (called the MDP with numerics).
\end{remark}

In our paper we extend these conclusions to certain non-Gaussian Hermitian Wigner matrices. 

\noindent
{\bf Tail-condition $(T)$:}
Say that $M_n$ satisfies the tail-condition $(T)$ if the real part $\eta$ and the imaginary part $\overline{\eta}$ of $M_n$ are independent and have 
a so-called stretched exponential decay: there
are two constants $C$ and $C'$ such that
$$
P \bigl( |\eta| \geq t^C \bigr) \leq e^{-t} \,\, \text{and} \,\, P \bigl( |\overline{\eta}| \geq t^C \bigr) \leq e^{-t}
$$
for all $t \geq C'$.
\vspace{0.5cm}

We say that two complex random variables $\eta_1$ and $\eta_2$ {\it match to order $k$} if
$$
\E \bigl[ \text{Re}(\eta_1)^m \, \text{Im}(\eta_1)^l \bigr] = \E \bigl[ \text{Re}(\eta_2)^m \, \text{Im}(\eta_2)^l \bigr]
$$
for all $m,l \geq 0$ such that $m+l \leq k$.

The following theorem is the main result of our paper:

\begin{theorem} \label{main}
Let $M_n$ be a Hermitian Wigner matrix whose entries satisfy tail-condition $(T)$ and match the corresponding entries of GUE up to order 4.
Set $W_n := \frac{1}{\sqrt{n}} M_n$. Then, for any $y \in (-2,2)$ and $I(y) = [y, \infty)$, with $Y_n := N_{I(y)}(W_n)$, for any
sequence $(a_n)_n$ of real numbers such that
\begin{equation} \label{regime2}
1 \ll a_n \ll \sqrt{\V (Y_n)}
\end{equation}
the sequence $(Z_n)_n$ with
$$
Z_n = \frac{Y_n - \E [Y_n]}{a_n \, \sqrt{\V (Y_n)}}
$$
satisfies a MDP with speed $a_n^2$ and rate function $I(x)=\frac{x^2}{2}$.
Moreover the sequence 
$$
\hat{Z}_n = \frac{Y_n - n \varrho_{sc}(I)}{a_n \, \sqrt{\frac{1}{2 \pi^2} \log n}}
$$
satisfies the MDP with the same speed, the same rate function, and in the same regime \eqref{regime2} (called the MDP with numerics).
\end{theorem}

Before we will prove the MDP for the GUE, we describe the organisation of the next sections. In a first step, we will apply
Theorem \ref{GUE-MDP1} to obtain a MDP for eigenvalues in the bulk of the semicircle law. Next we extend this result to certain families
of Hermitian Wigner matrices satisfying tail-condition $(T)$ by means of the Four Moment Theorem due to Tao and Vu. This is the content of Section 2. In Section 3, we show
the MDP with numerics for the counting function of Wigner matrices. Moreover we apply recent results of Erd\"os, Yau and Yin \cite{Erdoes/Yau/Yin:2010} 
on the localization of eigenvalues
and of Dallaporta and Vu \cite{Dallaporta/Vu:2011} in order to prove the MDP without numerics. Section 4 is devoted to discuss the case of symmetric real Wigner
matrices as well as the symplectic Gaussian ensemble applying interlacing formulas due to Forrester and Rains, \cite{Forrester/Rains:2001}. Finally, in Section
5 we present results for covariance matrices. We prove a universal MDP with numerics for the counting eigenvalue
function of covariance matrices.

In \cite{DoeringEichelsbacher:2010} we proved a MDP for certain determinantal point processes (DPP), including GUE.
Theorem \ref{GUE-MDP1} follows immediately from an improvement of Theorem 5.2 in \cite{DoeringEichelsbacher:2010}, which
can be easily observed applying the proof of \cite[Theorem 4.2.25]{Zeitounibook}. 
Let $\Lambda$ be a locally compact Polish space, equipped with a positive Radon measure $\mu$ on its
Borel $\sigma$-algebra. Let ${\mathcal M}_+(\Lambda)$ denote the set of positive $\sigma$-finite Radon measures on $\Lambda$.
A point process is a random, integer-valued $\chi \in {\mathcal M}_+(\Lambda)$, and it is simple if $P( \exists x \in \Lambda: \chi(\{x\}) >1)=0$.
Let $\chi$ be simple.
A locally integrable function $\varrho : \Lambda^k \to [0, \infty)$ is called a joint intensity (correlation), if for
any mutually disjoint family of subsets $D_1, \ldots, D_k$ of $\Lambda$
$$
\E \bigl( \prod_{i=1}^k \chi(D_i) \bigr) = \int_{\prod_{i=1}^k D_i} \varrho_k(x_1, \ldots, x_k) d\mu(x_1) \cdots d \mu(x_k),
$$
where $\E$ denotes the expectation with respect to the law of the point configurations of $\chi$.
A simple point process $\chi$ is said to be a {\it determinantal point process} with kernel $K$ if its joint intensities $\varrho_k$
exist and are given by
\begin{equation} \label{DPP}
\varrho_k(x_1, \ldots, x_k) = \det \bigl(  K(x_i, x_j) \bigr)_{i,j=1, \ldots,k}.
\end{equation}
An integral operator ${\mathcal K}: L^2(\mu) \to L^2(\mu)$ with kernel $K$ given by
$$
{\mathcal K}(f)(x) = \int K(x,y) f(y) \, d \mu(y), f \in L^2(\mu),
$$
is {\it admissible} with admissible kernel $K$ if ${\mathcal K}$ is self-adjoint, nonnegative and locally trace-class
(for details see \cite[4.2.12]{Zeitounibook}). A standard result is, that an integral compact operator
${\mathcal K}$ with admissible kernel $K$ possesses the decomposition
$
{\mathcal K} f(x) = \sum_{k=1}^n \eta_k \phi_k(x) \langle \phi_k, f \rangle_{L^2(\mu)},
$
where the functions $\phi_k$ are orthonormal in $L^2(\mu)$, $n$ is either finite or infinite, and $\eta_k >0$
for all $k$, leading to
\begin{equation} \label{kernelrep}
K(x,y) = \sum_{k=1}^n \eta_k \phi_k(x) \phi_k^*(y),
\end{equation}
an equality in $L^2(\mu \times \mu)$.
Moreover, an admissible integral operator ${\mathcal K}$ with kernel $K$ is called {\it good} with good kernel $K$ if the $\eta_k$ in \eqref{kernelrep}
satisfy $\eta_k \in (0,1]$. If the kernel $K$ of a determinantal point process is (locally) admissible, then it must in fact be good, see
\cite[4.2.21]{Zeitounibook}.

The following example is the main motivation for discussing determinantal point processes in this paper.
Let $(\lambda_1^n, \ldots, \lambda_n^n)$ be the eigenvalues of the GUE (Gaussian unitary ensemble) of dimension $n$ and denote
by $\chi_n$ the point process $\chi_n(D) = \sum_{i=1}^n 1_{\{ \lambda_i^n \in D\}}$. Then $\chi_n$ is a determinantal point process with admissible, good
kernel $K^{(n)}(x,y)= \sum_{k=0}^{n-1} \Psi_k(x) \Psi_k(y)$, where the functions $\Psi_k$ are the oscillator wave-functions, that is
$\Psi_k(x) := \frac{e^{-x^2/4} H_k(x)}{\sqrt{\sqrt{2 \pi} k!}}$, where $H_k(x):= (-1)^k e^{x^2/2} \frac{d^k}{dx^k} e^{-x^2/2}$ is the $k$-th
Hermite polynomial; see \cite[Def. 3.2.1, Ex. 4.2.15]{Zeitounibook}.

We will apply the following representation due to \cite[Theorem 7]{HoKPV06}: Suppose $\chi$ is a determinantal process with good kernel $K$ of the form
\eqref{kernelrep}, with $\sum_k \eta_k < \infty$. Let $(I_k)_{k=1}^n$ be independent Bernoulli variables with $P(I_k=1) = \eta_k$. Set
$
K_I(x,y) = \sum_{k=1}^n I_k \, \phi_k(x) \phi_k^*(y),
$
and let $\chi_I$ denote the determinantal point process with random kernel $K_I$. Then $\chi$ and $\chi_I$ have the same distribution, interpreted as stating that the mixture of determinental processes $\chi_I$ has the same distribution as $\chi$.
In the following let $K$ be a good kernel and for $D \subset \Lambda$ we write $K_D(x,y)= 1_D(x) K(x,y) 1_D(y)$. Let $D$ be such that
$K_D$ is trace-class, with eigenvalues $\eta_k$, $k \geq 1$. Then $\chi(D)$ has the same distribution as $\sum_k \xi_k$
where $\xi_k$ are independent Bernoulli random variables with $P(\xi_k=1)= \eta_k$ and $P(\xi_k =0) = 1 - \eta_k$.

\begin{theorem} \label{mdpDDP}
Consider a sequence $(\chi_n)_n$ of determinantal point processes on $\Lambda$ with good kernels $K_n$. Let $D_n$
be a sequence of measurable subsets of $\Lambda$ such that $(K_n)_{D_n}$ is trace-class. Assume that $(a_n)_n$ is a sequence of real numbers such that
$$
1 \ll a_n \ll \bigl( \sum_{k=1}^n \eta_k^n(1- \eta_k^n) \bigr)^{1/2},
$$
where $\eta_k^n$ are the eigenvalues of $K_n$.
Then $(Z_n)_n$ with
$$
Z_n := \frac{1}{a_n} \frac{\chi_n(D_n) - \E (\chi_n(D_n))}{\sqrt{\V (\chi_n(D_n))}}
$$
satisfies a moderate deviation principle with speed $a_n^2$ and rate function $I(x) = \frac{x^2}{2}$.
\end{theorem}

\begin{remark} In \cite{Merlevede/Peligrad:2009}, functional moderate deviations for triangular arrays of certain independent, not
identically distributed random variables are considered. Our result, Theorem \ref{mdpDDP}, seem to follow from Proposition 1.9 in \cite{Merlevede/Peligrad:2009}. Anyhow we prefer to present a direct proof.
\end{remark}

\begin{proof}[Proof of Theorem \ref{mdpDDP}]
We adapt the proof of \cite[Theorem 4.2.25]{Zeitounibook}. We write $K_n$ for the kernel $(K_n)_{D_n}$ and let $S_n := \sqrt{\V (\chi_n(D_n))}$.
$\chi_n(D_n)$ has the same distribution as the sum of independent Bernoulli variables $\xi_k^n$ whose parameters $\eta_k^n$ are the eigenvalues of $K_n$.
We obtain $S_n^2= \sum_k \eta_k^n(1- \eta_k^n)$ and since $K_n$ is trace-class we can write, for any $\theta$ real
\begin{eqnarray*}
\log \E \bigl[e^{\theta a_n^2 \, Z_n}\bigr] & = & \sum_k \log \E \biggl[ \exp \bigl( \frac{\theta a_n^2 (\xi_k^n-\eta_k^n)}{a_n S_n} \bigr) \biggr]\\
& = & -\frac{ \theta a_n^2 \sum_k \eta_k^n}{a_n S_n} + \sum_k \log \biggl( 1 + \eta_k^n \bigl( e^{a_n^2 \theta/ (a_n S_n)} -1\bigr) \biggr).
\end{eqnarray*}
For any real $\theta$ and $n$ large enough such that $\eta_k^n(e^{\theta a_n/S_n} -1) \in [0,1]$ we apply Taylor for $\log (1+x)$ and obtain
$$
\frac{1}{a_n^2} \log \E \bigl[e^{\theta a_n^2 Z_n} \bigr] =  \frac{ \theta^2 a_n^2 \sum_k \eta_k^n(1-\eta_k^n)}{2 a_n^2 S_n^2} + o 
\biggl(\frac{a_n^3 \, \sum_k \eta_k^n(1-\eta_k^n)}{a_n^2 S_n^3} \biggr).
$$
The last term is $o \bigl(\frac{a_n}{S_n} \bigr)$.
Applying the Theorem of G\"artner-Ellis, \cite[Theorem 2.3.6]{Dembo/Zeitouni:LargeDeviations}, the result follows.
\end{proof}

\begin{proof}[Proof of Theorem \ref{GUE-MDP1}]
Now the first statement of Theorem \ref{GUE-MDP1} follows since $\V( \chi_n(D_n)) \to \infty$, see \cite[Cor. 4.2.27]{Zeitounibook}.
In particular for any $I=[y, \infty)$ with $y \in (-2,2)$
$$
\tilde{Z}_n :=\frac{N_I(W_n') - \E[N_I(W_n')]}{a_n \sqrt{\V(N_I(W_n'))}}
$$
satisfies the MDP. The MDP with numerics (see Remark \ref{renew}) follows, since the sequences $(\tilde{Z}_n)_n$ and $(\hat{Z}_n)_n$ are exponentially equivalent in the sense of  definition
\cite[Definition 4.2.10]{Dembo/Zeitouni:LargeDeviations},
and hence the result follows from \cite[Theorem 4.2.13]{Dembo/Zeitouni:LargeDeviations}: Let
$$
Z_n'' := \frac{N_I(W_n') - n \varrho_{sc}(I)}{a_n \, \sqrt{ \V(N_I(W_n'))}}.
$$
Since $|\tilde{Z}_n - Z_n''| = \big| \frac{\E[N_I(W_n')] - n \varrho_{sc}(I)}{a_n \sqrt{ \V(N_I(W_n'))}} \big| \to 0$ as $n \to \infty$,   
$(\tilde{Z}_n)_n$ and $(Z_n'')_n$ 
are exponentially equivalent. Moreover by Taylor we obtain
$
|Z_n'' - \hat{Z}_n| = \frac{o(1)}{a_n} \frac{N_I(W_n') - n \varrho_{sc}(I)}{\sqrt{\V(N_I(W_n'))}}
$
and the MDP for $(Z_n'')_n$ implies 
$
\limsup_{n \to \infty} \frac{1}{a_n^2} \log P \bigl( |Z_n'' - \hat{Z}_n| > \varepsilon \bigr) = - \infty
$
for any $\varepsilon >0$.
\end{proof}

\section{Moderate deviations for eigenvalues in the bulk}

Under certain conditions on $i$ it was proved in \cite{Gustavsson:2005} that the $i$-th eigenvalue $\lambda_i$ of the GUE $W_n'$ satisfies a CLT.
Consider $t(x) \in [-2,2]$ defined for $x \in [0,1]$ by
$$
x = \int_{-2}^{t(x)} d \varrho_{sc}(t) = \frac{1}{2 \pi} \int_{-2}^{t(x)} \sqrt{4 - x^2} \, dx.
$$
Then for $i=i(n)$ such that $i/n \to a \in (0,1)$ as $n \to \infty$ (i.e. $\lambda_i$ is eigenvalue in the bulk), $\lambda_i(W_n')$ satisfies a CLT:
\begin{equation} \label{CLT-G}
\sqrt{\frac{4 - t(i/n)^2}{2}} \frac{\lambda_i(W_n') - t(i/n)}{\frac{\sqrt{\log n}}{n}} \to N(0,1)
\end{equation}
for $n \to \infty$. Remark that $t(i/n)$ is sometimes called the {\it classical or expected location} of the $i$-th eigenvalue. The standard deviation
is $\frac{\sqrt{\log n}}{\pi \sqrt{2}} \, \frac{1}{n \varrho_{sc}(t(i/n))}$. Note that from the semicircular law, the factor $\frac{1}{n \varrho_{sc}(t(i/n))}$ is the mean
eigenvalue spacing.

The proof in \cite{Gustavsson:2005} is achieved by the tight relation between eigenvalues and the counting function
expressed by the elementary equivalence, for $I(y)=[y, \infty)$, $y \in \R$,
\begin{equation} \label{relation}
N_{I(y)}(W_n) \leq n-i \,\, \text{if and only if} \,\, \lambda_i \leq y.
\end{equation}
Hence the theorem due to Costin and Lebowitz as well as Soshnikov, see \cite{Soshnikov:2000}, can be applied. Moreover the proof
in \cite{Gustavsson:2005} relies on fine asymptotics for the Airy function and the Hermite polynomials due to \cite{Deift/Thomas:1999}.

Our first result in this Section is a corresponding MDP for $\lambda_i$ in the bulk:

\begin{theorem} \label{MDP-GUE-Gust}
Consider the GUE matrix $W_n' = \frac{1}{\sqrt{n}} M_n'$. Consider $i=i(n)$ such that $i/n \to a \in (0,1)$ as $n \to \infty$. If $\lambda_i$ denotes the
eigenvalue number $i$ in the GUE matrix $W_n'$ it holds that for any sequence $(a_n)_n$ of real numbers such that
$1 \ll a_n \ll \sqrt{\log n}$ the sequence $(X_n')_n$ with
$$
X_n' = \sqrt{\frac{4 - t(i/n)^2}{2}} \frac{\lambda_i(W_n') - t(i/n)}{a_n \, \frac{\sqrt{\log n}}{n}}
$$
satisfies a MDP with speed $a_n^2$ and rate function $I(x)=\frac{x^2}{2}$.
\end{theorem}

Interesting enough, this result can be extended to large families of Hermitian Wigner matrices satisfying tail-condition $(T)$
by means of the Four Moment Theorem of Tao and Vu:

\begin{theorem} \label{uni1}
Consider a Hermitian Wigner matrix $W_n = \frac{1}{\sqrt{n}} M_n$ whose entries satisfy tail-condition $(T)$ and match
the corresponding entries of GUE up to order 4. Consider $i=i(n)$ such that $i/n \to a \in (0,1)$ as $n \to \infty$. If $\lambda_i$ denotes the
eigenvalue number $i$ of $W_n$ it holds that for any sequence $(a_n)_n$ of real numbers such that
$1 \ll a_n \ll \sqrt{\log n}$, the sequence $(X_n)_n$ with
$$
X_n = \sqrt{\frac{4 - t(i/n)^2}{2}} \frac{\lambda_i(W_n) - t(i/n)}{a_n \, \frac{\sqrt{\log n}}{n}}
$$
satisfies a MDP with speed $a_n^2$ and rate function $I(x)=\frac{x^2}{2}$.
\end{theorem}

\begin{remark} In \cite{Gustavsson:2005} a CLT at the edge of the spectrum was considered also. The proof applies 
the result of Costin and Lebowitz as well as fine asymptotics presented in \cite{Deift/Thomas:1999}. Consider
$i \to \infty$ such that $i/n \to 0$ as $n \to \infty$ and consider $\lambda_{n-i}$, eigenvalue number $n-i$ in the
GUE or Hermitian Wigner case. A CLT for the rescaled $\lambda_{n-i}$ is stated in \cite[Theorem 1.2]{Gustavsson:2005}.
We would be able to formulate and prove a MDP for eigenvalue $\lambda_{n-i}$, but it is not the main focus
of this paper. We omit this.
\end{remark}

\begin{proof}[Proof of Theorem \ref{MDP-GUE-Gust}]
The proof is oriented to the proof of \cite[Theorem 1.1]{Gustavsson:2005} and will apply the precise asymptotic behaviour of the expectation and of the variance
of the counting function $N_I(W_n')$, see \eqref{asymp}, which is a reformulation of \cite[Lemma 2.1-2.3]{Gustavsson:2005}.
Let $P_n$ denote the probability of the GUE determinantal point processes, and set
$$
I_n := \biggl[ t(i/n) + \xi \, a_n \frac{\sqrt{\log n}}{n} \frac{\sqrt{2}}{\sqrt{4 - t(i/n)^2}}, \infty \biggr).
$$
Now we apply relation \eqref{relation} and obtain
\begin{eqnarray*}
& & P_n \biggl( \frac{\lambda_i(W_n') - t(i/n)}{a_n \frac{\sqrt{\log n}}{n} \frac{\sqrt{2}}{\sqrt{4 - t(i/n)^2}}} \leq \xi \biggr) =  
P_n \biggl( \lambda_i(W_n') \leq \xi \, a_n \, \frac{\sqrt{\log n}}{n} \frac{\sqrt{2}}{\sqrt{4 - t(i/n)^2}} + t(i/n) \biggr) \\
& = & P_n \bigl( N_{I_n}(W_n') \leq n-i \bigr) = P_n \biggl( \frac{N_{I_n}(W_n') - \E[N_{I_n}(W_n')]}{a_n \, (\V(N_{I_n}(W_n')))^{1/2}} \leq 
\frac{n-i - \E[N_{I_n}(W_n')]}{a_n \, (\V(N_{I_n}(W_n')))^{1/2}} \biggr).
\end{eqnarray*}
Since $i/n \to a \in (0,1)$, by definition of $t(x)$ we obtain $t(i/n) \in (-2,2)$. Moreover since $a_n \ll \sqrt{\log n}$ we have
$\xi \, a_n \, \frac{\sqrt{\log n}}{n} \frac{\sqrt{2}}{\sqrt{4 - t(i/n)^2}} + t(i/n) \in (-2,2)$ for $n$ large. Therefore with \eqref{asymp}
we have for $n$ sufficiently large
that
$$
\E [ N_{I_n}(W_n')] = n \, \varrho_{sc}(I_n) + O \bigl(\frac{\log n}{n} \bigr) \,\, \text{and} \,\, \V(N_{I_n}(W_n')) = \bigl( \frac{1}{2\pi^2} + o(1) \bigr) \log n.
$$
With $b(n) :=  a_n \, \frac{\sqrt{\log n}}{n} \frac{\sqrt{2}}{\sqrt{4 - t(i/n)^2}}$ and $f_n(t(i/n)) := t(i/n) + \xi \, b(n)$ we get from symmetry
$$
\varrho_{sc}(I_n) = \int_{f_n(t(i/n))}^{\infty} \varrho_{sc}(x) \, dx = \frac 12 - \int_{0}^{f_n(t(i/n))} \varrho_{sc}(x) \, dx = 1 - \int_{-2}^{f_n(t(i/n))} \varrho_{sc}(x) \, dx.
$$
Now
$$
\int_{-2}^{f_n(t(i/n))} \varrho_{sc}(x) \, dx = \int_{-2}^{t(i/n)}  \varrho_{sc}(x) \, dx + \int_{t(i/n)}^{f_n(t(i/n))} \varrho_{sc}(x) \, dx = \frac in + 
\int_{t(i/n)}^{f_n(t(i/n))} \varrho_{sc}(x) \, dx
$$
and
$$
\int_{t(i/n)}^{f_n(t(i/n))} \varrho_{sc}(x) \, dx = \xi b(n) \frac{1}{2 \pi} \sqrt{4 - t(i/n)^2} + O \bigl(b(n)^2).
$$
Summarizing we obtain
$$
n \, \varrho_{sc}(I_n) = n-i - \xi \, a_n \, \sqrt{ \log n} \frac{1}{\sqrt{2} \pi} + O \bigl( \frac{a_n^2 \log n}{n} \bigr)
$$
and therefore
$$
\frac{n-i - \E[N_{I_n}(W_n')]}{a_n \, (\V(N_{I_n}(W_n')))^{1/2}} = \xi + \varepsilon(n),
$$
where $\varepsilon(n) \to 0$ as $n \to \infty$. By Theorem \ref{GUE-MDP1} we obtain for every $\xi < 0$
\begin{equation} \label{r1}
\lim_{n \to \infty} \frac{1}{a_n^2} \log P_n  \biggl( \frac{\lambda_i(W_n') - t(i/n)}{a_n \frac{\sqrt{\log n}}{n} \frac{\sqrt{2}}{\sqrt{4 - t(i/n)^2}}} \leq \xi \biggr) 
= - \frac{\xi^2}{2}.
\end{equation}
With
$$
 P_n  \biggl( \frac{\lambda_i(W_n') - t(i/n)}{a_n \frac{\sqrt{\log n}}{n} \frac{\sqrt{2}}{\sqrt{4 - t(i/n)^2}}} \geq \xi \biggr) = P_n \bigl( N_{I_n}(W_n') \geq n-i+1 \bigr)
$$
the same calculations lead, for every $\xi >0$, to
\begin{equation} \label{r2}
\lim_{n \to \infty} \frac{1}{a_n^2} \log P_n  \biggl( \frac{\lambda_i(W_n') - t(i/n)}{a_n \frac{\sqrt{\log n}}{n} \frac{\sqrt{2}}{\sqrt{4 - t(i/n)^2}}} \geq \xi \biggr) 
= - \frac{\xi^2}{2}.
\end{equation}
Hence the conclusion follows: see for example \cite[Proof of Theorem 2.2.3]{Dembo/Zeitouni:LargeDeviations}. To be more precise we apply the preceding results \eqref{r1} 
and \eqref{r2} with Theorem 4.1.11 in \cite{Dembo/Zeitouni:LargeDeviations}, which allows us to derive a MDP from the limiting behaviour of probabilities for a basis
of topology. For the latter, we choose all open intervals $(a,b)$, where at least one of the endpoints is finite and where none of the endpoints is zero. Denote
the family of such intervals by ${\mathcal U}$. From \eqref{r1} and \eqref{r2}, it follows for each $U=(a,b) \in {\mathcal U}$,
$$
{\mathcal L}_{U}: = \lim_{n \to \infty} \frac{1}{a_n^2} \log P \bigl( X_n' \in U \bigr) = \left\{ \begin{array}{r@{\quad:\quad}l}
b^2/2 & a < b < 0 \\ 0 & a < 0 < b \\ a^2/2 & 0<a<b \end{array} \right. 
$$
By \cite[Theorem 4.1.11]{Dembo/Zeitouni:LargeDeviations}, $(X_n')_n$ satisfies a weak MDP with speed $a_n^2$ and rate function 
$$
t \mapsto \sup_{U \in {\mathcal U}; t \in U} {\mathcal L}_U = \frac{t^2}{2}.
$$
With \eqref{r2}, it follows that $(X_n')_n$ is exponentially tight, hence by Lemma 1.2.18 in \cite{Dembo/Zeitouni:LargeDeviations}, $(X_n')_n$ satisfies the MDP
with the same speed and the same good rate function. This completes the proof.
\end{proof}

To extend the result of Theorem \ref{MDP-GUE-Gust} to Hermitian Wigner matrices satisfying tail-condition $(T)$, we will apply the Four Moment Theorem (for the bulk of
the spectra of Wigner matrices), see \cite[Theorem 15]{Tao/Vu:2009}.

\begin{theorem}[Four Moment Theorem due to Tao and Vu]
There is a small positive constant $c_0$ such that for every $0 < \varepsilon < 1$ and $k \geq 1$ the following holds. Let $M_n$ and $M_n'$ be two Hermitian Wigner
matrices satisfying tail-condition $(T)$. Assume furthermore that for any $1 \leq i <j \leq n$, $Z_{ij}$ and $Z_{ij}'$ match to order 4 and for any $1 \leq i \leq n$, $Y_i$
and $Y_i'$ match to order 2. Set $A_n :=\sqrt{n} M_n$ and $A_n' := \sqrt{n} M_n'$, and let $G : {\Bbb R}^k \to {\Bbb R}$ be a smooth function obeying the derivative bounds 
$|\nabla^jG(x)| \leq n^{c_0}$ for all $0 \leq j \leq 5$ and $x \in {\Bbb R}^k$. Then for any $\varepsilon n \leq i_1 < i_2 \cdots < i_k \leq (1- \varepsilon)n$, and
for $n$ sufficiently large depending on $\varepsilon$, $k$ and the constants $C, C'$ in tail-condition $(T)$, we have
\begin{equation} \label{taovu}
|\E \bigl( G(\lambda_{i_1}(A_n), \ldots, \lambda_{i_k}(A_n) ) \bigr) - \E \bigl( G(\lambda_{i_1}(A_n'), \ldots, \lambda_{i_k}(A_n') ) \bigr)| \leq n^{-c_0}.
\end{equation}
\end{theorem}

Applying this Theorem for the special case when $M_n'$ is GUE, one obtains \cite[Corollary 18]{Tao/Vu:2009}:

\begin{corollary} \label{taovu-cor}
Let $M_n$ be a Hermitian Wigner matrix whose atom distribution $\xi$ satisfies $\E \xi^3=0$ and $\E \xi^4 = \frac 34$ and tail-condition $(T)$, and $M_n'$ be a random matrix
sampled from GUE. Then with $G, A_n, A_n'$ as in the previous theorem, and $n$ sufficiently large, one has
\begin{equation} \label{taovu2}
|\E \bigl( G(\lambda_{i_1}(A_n), \ldots, \lambda_{i_k}(A_n) ) \bigr) - \E \bigl( G(\lambda_{i_1}(A_n'), \ldots, \lambda_{i_k}(A_n') ) \bigr)| \leq n^{-c_0}.
\end{equation}
\end{corollary}

Now the universality of the MDP in Theorem \ref{uni1} follows along the lines of the proof of \cite[Corollary 21]{Tao/Vu:2009}.

\begin{proof}[Proof of Theorem \ref{uni1}]
Let $M_n$ be a Hermitian Wigner matrix whose entries satisfy tail-condition $(T)$ and match the corresponding entries of GUE up to order 4.
Let $i, a$ and $t(\cdot)$ be as in the statement of the Theorem, and let $c_0$ be as in Corollary \ref{taovu-cor}. Then \cite[(18)]{Tao/Vu:2009}
says that
\begin{equation} \label{inequ}
P_n \bigl( \lambda_i(A_n') \in I_{-} \bigr) - n^{-c_0} \leq P_n \bigl( \lambda_i(A_n) \in I \bigr) \leq P_n \bigl( \lambda_i(A_n') \in I_{+} \bigr) + n^{-c_0}
\end{equation}
for all intervals $[b,c]$, and $n$ sufficiently large depending on $i$ and the constants $C, C'$ of tail-condition $(T)$. Here $I_{+} := [b-n^{-c_0/10}, c+n^{-c_0/10}]$
and $I_{-} := [b+n^{-c_0/10}, c-n^{-c_0/10}]$. We present the argument of proof of \eqref{inequ} just to make the presentation more self-contained.
One can find a smooth bump function $G : {\mathbb R} \to {\mathbb R}_+$ which is equal to one on the smaller interval $I_{-}$ and vanishes outside
the larger interval $I_+$. It follows that $P_n \bigl( \lambda_i(A_n) \in I \bigr) \leq \E G(\lambda_i(A_n))$ and 
$\E G(\lambda_i(A_n')) \leq P_n \bigl( \lambda_i(A_n') \in I \bigr)$. One can choose $G$ to obey the condition $|\nabla^j G(x)| \leq n^{c_0}$ for $j=0, \ldots, 5$
and hence
by Corollary \ref{taovu-cor} one gets
$$
| \E G(\lambda_i(A_n)) - \E G(\lambda_i(A_n'))| \leq n^{-c_0}.
$$
Therefore the second inequality in \eqref{inequ} follows from the triangle inequality. The first inequality is proven similarly.

Now for $n$ sufficiently large we consider the interval $I_n := [b_n, c_n]$ with
$$
b_n :=  b \, a_n \sqrt{\log n} \frac{\sqrt{2}}{\sqrt{4 - t(i/n)^2}} + n t(i/n) \,\, \text{and} \,\, c_n :=  c \, a_n \sqrt{\log n} \frac{\sqrt{2}}{\sqrt{4 - t(i/n)^2}} + n t(i/n)
$$
with $b,c \in {\mathbb R}$,  $b \leq c$. Then for $X_n$ defined as in the statement of the Theorem we have
$$
P_n \bigl( X_n \in [b,c] \bigr) = P_n \biggl( \frac{\lambda_i(A_n) - n t(i/n)}{a_n \sqrt{\log n}\frac{\sqrt{2}}{\sqrt{4 - t(i/n)^2}}} \in [b,c] \biggr) 
= P_n \bigl( \lambda_i(A_n) \in I_n \bigr).
$$
With \eqref{inequ} and \cite[Lemma 1.2.15]{Dembo/Zeitouni:LargeDeviations} we obtain
$$
\limsup_{n \to \infty} \frac{1}{a_n^2} \log P_n \bigl( X_n \in [b,c] \bigr) \leq \max \biggl( \limsup_{n \to \infty} \frac{1}{a_n^2} \log P_n \bigl( \lambda_i(A_n') \in (I_n)_+
\bigr) ; \limsup_{n \to \infty} \frac{1}{a_n^2} \log n^{-c_0} \biggr).
$$
For the first object we have
$$
\limsup_{n \to \infty} \frac{1}{a_n^2} \log P_n \bigl( \lambda_i(A_n') \in (I_n)_+ \bigr)  = \limsup_{n \to \infty} \frac{1}{a_n^2} \log P_n \biggl( 
\frac{\lambda_i(A_n') - n t(i/n)}{a_n \sqrt{\log n}\frac{\sqrt{2}}{\sqrt{4 - t(i/n)^2}}} \in [b - \eta(n), c + \eta(n)] \biggr)
$$
with $\eta(n) = n^{-c_0/10} \bigl( a_n \sqrt{\log n} \frac{\sqrt{2}}{\sqrt{4 - t(i/n)^2}} \bigr)^{-1} \to 0$ as $n \to \infty$.
Since $c_0 >0$ and $ \log n / a_n^2 \to \infty$ for $n \to \infty$ by assumption, applying Theorem \ref{MDP-GUE-Gust} we have
$$
\limsup_{n \to \infty} \frac{1}{a_n^2} \log P_n \bigl( X_n \in [b,c] \bigr) \leq - \inf_{x \in [b,c]} \frac{x^2}{2}.
$$
Applying the first inequality in \eqref{inequ} in the same manner we also obtain the lower bound
$$
\liminf_{n \to \infty} \frac{1}{a_n^2} \log P_n \bigl( X_n \in [b,c] \bigr) \geq - \inf_{x \in [b,c]} \frac{x^2}{2}.
$$
Finally the argument in the last part of the proof of Theorem \ref{MDP-GUE-Gust} can be repeated to obtain the MDP for $(X_n)_n$.
\end{proof}

\section{Universality of the Moderate deviations}

In this section we proof Theorem \ref{main}. As announced, we will show that the MDP behaviour of eigenvalues in the bulk of the GUE (Theorem \ref{MDP-GUE-Gust})
extended to Hermitian Wigner matrices (Theorem \ref{uni1}) leads to the MDP with numerics for the counting function of eigenvalues of Hermitian Wigner matrices. 

\begin{proof}[Proof of Theorem \ref{main}]
For every $\xi \in {\mathbb R}$ we obtain that
$$
P_n \bigl( \hat{Z}_n \leq \xi \bigr) = P_n \bigl( N_{I(y)}(W_n) \leq n - i_n \bigr)
$$
with $i_n := n \varrho_{sc}((-\infty,y]) - \xi \, a_n \, \sqrt{\frac{1}{2 \pi^2} \log n}$. Hence using \eqref{relation} it follows
$$
P_n \bigl( \hat{Z}_n \leq \xi \bigr) = P_n \bigl( \lambda_{i_n} \leq y \bigr) = P_n \biggl( \sqrt{\frac{4- t(i_n/n)^2}{2}} \frac{\lambda_{i_n}(W_n) - t(i_n/n)}{a_n 
\frac{\sqrt{ \log n}}{n}} \leq \xi_n \biggr)
$$
with
$\xi_n := \sqrt{\frac{4- t(i_n/n)^2}{2}} \frac{y- t(i_n/n)}{a_n \frac{\sqrt{ \log n}}{n}}$. Now
$$
\frac{i_n}{n} = \varrho_{sc}((-\infty,y]) - \frac{\xi a_n \sqrt{\frac{1}{2 \pi^2} \log n}}{n} \to \varrho_{sc}((-\infty,y]) \in (0,1)
$$
for $n \to \infty$. We will prove that $\xi_n = \xi + o(1)$. Applying Theorem \ref{uni1}, it follows that
$$
\lim_{n \to \infty} \frac{1}{a_n^2} \log P_n \bigl( \hat{Z}_n \leq \xi \bigr) = -\frac{\xi^2}{2}
$$
for all $\xi < 0$. With $P_n \bigl( \hat{Z}_n \geq \xi \bigr) = P_n \bigl(N_{I(y)}(W_n) \geq n - i_n \bigr) = P_n \bigl( \lambda_{i_n+1} \geq y \bigr)$
the same calculations will lead, for any $\xi >0$, to
$$
\lim_{n \to \infty} \frac{1}{a_n^2} \log P_n \bigl( \hat{Z}_n \geq \xi \bigr) = -\frac{\xi^2}{2}.
$$
The MDP for $(\hat{Z}_n)_n$ (the MDP with numerics) now follows along the lines of the proof of Theorem \ref{MDP-GUE-Gust} (topological argument). 

The MDP for
$(Z_n)_n$ follows by the arguments given in the proof of Theorem \ref{GUE-MDP1}, using the deep fact that the expectation $\E[Y_n]$ and the variance $\V(Y_n)$
of the eigenvalue counting function have identical behaviours to the ones for GUE matrices:
$$
\E[Y_n] = n \varrho_{sc}(I(y)) + o(1)\,\, \text{and}\,\, \V(Y_n) = \bigl( \frac{1}{2 \pi^2} + o(1) \bigr) \log n.
$$
This result is established in \cite[Theorem 2]{Dallaporta/Vu:2011}, applying strong localization of the eigenvalues of Wigner matrices, a recent result
from \cite{Erdoes/Yau/Yin:2010}. 

Finally we prove that $\lim_{n \to \infty} \xi_n = \xi$. We obtain
\begin{eqnarray*}
t(i_n/n) & = & t \biggl( 
\varrho_{sc}((-\infty,y]) - \frac 1n \, \xi \, a_n \, \sqrt{\frac{1}{2 \pi^2} \log n} \biggr) \\
& = & t \bigl( \varrho_{sc}((-\infty,y]) \bigr) - \frac 1n \, \xi \, a_n \, \sqrt{\frac{1}{2 \pi^2} \log n} \, t' \bigl( \varrho_{sc}((-\infty,y]) \bigr) + o \bigl( \frac{a_n \sqrt{\log n}}{n} \bigr) \\
& = & y - \xi \frac 1n a_n \sqrt{\log n} \frac{\sqrt{2}}{\sqrt{4-y^2}} + o \bigl( \frac{a_n \sqrt{\log n}}{n} \bigr).
\end{eqnarray*}
Hence $\frac{y - t(i_n/n)}{ a_n \frac{\sqrt{\ log n}}{n}} = \xi \frac{\sqrt{2}}{\sqrt{4-y^2}} + o(1)$ and with $\lim_{n \to \infty} \sqrt{\frac{4- t(i_n/n)^2}{2}} = 
\sqrt{\frac{4- y^2}{2}}$ it follows that $\lim_{n \to \infty} \xi_n = \xi$.
\end{proof}

\section{Symmetric Wigner matrices and the GSE}

In this section, we indicate how the preceding results for Hermitian Wigner matrices can be stated and proved for real Wigner symmetric
matrices. Moreover we consider a Gaussian symplectic ensemble
(GSE). Real Wigner matrices are random symmetric matrices $M_n$ of size $n$ such that, for $i<j$, $(M_n)_{ij}$ are i.i.d. with mean zero and variance one, $(M_n)_{ii}$
are i.i.d. with mean zero and variance 2. As already mentioned, the case where the entries are Gaussian is the GOE. As in Section 1 and 2, the main
issue is to establish our conclusions for the GOE. On the level of CLT, this was developed in \cite{Rourke:2010} by means of the famous {\it interlacing formulas}
due to Forrester and Rains, \cite{Forrester/Rains:2001}, that relates the eigenvalues of different matrix ensembles.
The following relation holds between matrix ensembles:
\begin{equation} \label{forrai}
{\rm GUE}_n = {\rm even} \bigl( \rm{GOE}_n \cup {\rm GOE}_{n+1} \bigr).
\end{equation} 

The statement is: Take two independent (!) matrices from the GOE: one of size $n \times n$ and one of size $(n+1) \times (n+1)$. Superimpose
the $2n+1$ eigenvalues on the real line and then take the $n$ even ones. They have the same distribution as the eigenvalues of a $n \times n$
matrix from the GUE. If $M_n^{\R}$ denotes a GOE matrix and $W_n^{\R} := \frac{1}{\sqrt{n}} M_n^{\R}$, 
first we will prove a MDP for 
\begin{equation} \label{GOEZn}
Z_n^{\R} := \frac{N_{I_n}(W_n^{\R}) - \E[N_{I_n}(W_n^{\R})]}{a_n \sqrt{\V(N_{I_n}(W_n^{\R}))}}
\end{equation}
for any $1 \ll a_n \ll \sqrt{\V(N_{I_n}(W_n^{\R}))}$, $I_n$ an interval in ${\mathbb R}$, with speed $a_n^2$ and rate $x^2/2$. 
Let within this section $M_n^{\C}$ denote a GUE matrix and $W_n^{\C}$ the corresponding normalized matrix. 
The nice consequences of \eqref{forrai} were already suitably developed in \cite{Rourke:2010}: applying Cauchy's interlacing theorem one can write
\begin{equation} \label{interl}
N_{I_n}(W_n^{\C}) = \frac 12 \bigl[  N_{I_n}(W_n^{\R}) +  N_{I_n}'(W_n^{\R}) + \eta_n'(I_n) \bigr],
\end{equation}
where one obtains ${\rm GOE}_n'$ in $N_{I_n}'(W_n^{\R})$ from ${\rm GOE}_{n+1}$ by considering the principle sub-matrix of ${\rm GOE}_{n+1}$
and $\eta_n'(I_n)$ takes values in $\{-2,-1,0,1,2\}$. Note that  $N_{I_n}(W_n^{\R})$ and $N_{I_n}'(W_n^{\R})$ are independent because ${\rm GOE}_{n+1}$\
and ${\rm GOE}_{n}$ denote independent matrices from the GOE. We obtain
\begin{eqnarray} \label{fein}
Z_n^{\C} := \frac{N_{I_n}(W_n^{\C}) - \E[N_{I_n}(W_n^{\C})]}{a_n \sqrt{\V(N_{I_n}(W_n^{\C}))}} & = & 
\frac{N_{I_n}(W_n^{\R}) - \E[N_{I_n}(W_n^{\R})]}{a_n 2 \sqrt{\V(N_{I_n}(W_n^{\C}))}} +  
\frac{N_{I_n}'(W_n^{\R}) - \E[N_{I_n}'(W_n^{\R})]}{a_n 2 \sqrt{\V(N_{I_n}(W_n^{\C}))}}  \nonumber \\
& + & \frac{\eta_n'(I_n) - \E[\eta_n'(I_n)]}{a_n 2  \sqrt{\V(N_{I_n}(W_n^{\C}))}} =: X_n +Y_n + \varepsilon_n 
\end{eqnarray}
Now we can make use of the MDP for $(Z_n^{\C})_n$, Theorem \ref{GUE-MDP1}.
Using the independence of $X_n$ and $Y_n$ in \eqref{fein},
as well as the fact that the third summand can be estimated by
$
|\varepsilon_n| \leq \frac{2}{a_n   \sqrt{\V(N_{I_n}(W_n^{\C}))}},
$
we obtain for every $\theta$
$$
-\frac{|2\theta|}{a_n \sqrt{\V(N_{I_n}(W_n^{\C}))}} + 2 \frac{1}{a_n^2} \log \E e^{\theta a_n^2 X_n} \leq \frac{1}{a_n^2} \log \E e^{\theta a_n^2 Z_n} 
\leq \frac{|2\theta|}{a_n \sqrt{\V(N_{I_n}(W_n^{\C}))}} + 2 \frac{1}{a_n^2} \log \E e^{\theta a_n^2 X_n}.
$$
Applying the Theorem of G\"artner-Ellis, \cite[Theorem 2.3.6]{Dembo/Zeitouni:LargeDeviations},
the MDP for $(X_n)_n$ with speed $a_n^2$ and rate $x^2$ follows for all $(I_n)_n$ with $\V(N_{I_n}(W_n^{\C})) \to \infty$. 
Hence we have proved the MDP-version of \cite[Lemma 2]{Rourke:2010}, 
that $\bigl( \frac{N_{I_n}(W_n^{\R}) - \E[N_{I_n}(W_n^{\R})]}{a_n \sqrt{2} \sqrt{\V(N_{I_n}(W_n^{\C}))}}\bigr)_n$ 
satisfies an MDP with rate $x^2/2$ if $\V(N_{I_n}(W_n^{\C})) \to \infty$.
The interlacing formula \eqref{interl} leads to $2 \V(N_{I_n}(W_n^{\C})) +O(1) = \V(N_{I_n}(W_n^{\R}))$ if $\V(N_{I_n}(W_n^{\C})) \to \infty$. Therefore
$(Z_n^{\R})_n$ satisfies the MDP with speed $a_n^2$ and rate function $x^2/2$.
The proof of Theorem \ref{MDP-GUE-Gust} can simply be adapted to the GOE.
Since the Four Moment Theorem also applies for real symmetric matrices, with an analog of \cite[Lemma 5]{Dallaporta/Vu:2011} in hand we obtain:


\begin{theorem} \label{uni2}
Consider a real symmetric Wigner matrix $W_n = \frac{1}{\sqrt{n}} M_n$ whose entries satisfy tail-condition $(T)$ and match the corresponding entries
of GOE up to order 4. 

{\bf(1)} Consider $i=i(n)$ such that $i/n \to a \in (0,1)$ as $n \to \infty$. Denote the eigenvalue number $i$ of $W_n$ by $\lambda_i$.
Let $(a_n)_n$ be a sequence of real numbers such that $1 \ll a_n \ll \sqrt{\log n}$.
Then the sequence $(X_n)_n$ with $X_n :=\frac{\lambda_i -t(i/n)}{a_n \frac{\sqrt{\log n}}{n} \frac{2}{\sqrt{4 - t(i/n)^2}}}$
universally satisfies a MDP with speed $a_n^2$ and rate function $I(x)=\frac{x^2}{2}$.
{\bf(2)} For any $y \in (-2,2)$ and $I(y)=[y, \infty)$, the rescaled eigenvalue counting function $N_{I(y)}(W_n)$ universally satisfies the MDP
as in Theorem \ref{main}.
\end{theorem}

Finally we consider the Gaussian Symplectic Ensemble (GSE). 
Here the following relation holds between matrix ensembles:
$ 
{\rm GSE}_n = {\rm even} \bigl({\rm GOE}_{2n+1} \bigr) \frac{1}{\sqrt{2}}.
$ 
The multiplication by $\frac{1}{\sqrt{2}}$ denotes scaling the $(2n+1) \times (2n+1)$ GOE matrix by the factor $\frac{1}{\sqrt{2}}$.
Let $x_1 < x_2 < \cdots < x_n$ denote the ordered eigenvalues of an $n \times n$ matrix from the GSE and let $y_1 <y_2 < \cdots < y_{2n+1}$ denote
the ordered eigenvalues of an $(2n+1) \times (2n+1)$ matrix from the GOE. Then it follows that $x_i = y_{2i}/\sqrt{2}$ in distribution. Hence
the MDP for the $i$-th eigenvalue of the GSE follows directly from the GOE case.

\begin{theorem} \label{MDP-GSE-Gust}
Consider the GSE matrix $W_n^{{\mathbb H}} = \frac{1}{\sqrt{n}} M_n^{{\mathbb H}}$. Consider $i=i(n)$ such that $i/n \to a \in (0,1)$ as $n \to \infty$. 
If $\lambda_i$ denotes the eigenvalue number $i$ in the GSE matrix, it holds that for any sequence $(a_n)_n$ of real numbers such that
$1 \ll a_n \ll \sqrt{\log n}$ the sequence 
$$
\sqrt{4 - t(i/n)^2} \frac{\lambda_i(W_n^{{\mathbb H}}) - t(i/n)}{a_n \, \frac{\sqrt{\log n}}{n}}
$$
satisfies a MDP with speed $a_n^2$ and rate function $I(x)=\frac{x^2}{2}$.
For any $y \in (-2,2)$ and $I(y)=[y, \infty)$, the rescaled eigenvalue counting function $N_{I(y)}(W_n^{{\mathbb H}})$ satisfies the MDP
as in Theorem \ref{GUE-MDP1}.
\end{theorem}



\section{Moderate deviations for covariance matrices}
In this section we briefly present the analogous results for covariance matrices. They rely on \cite{Su:2006}, where Gaussian fluctuations of individual
eigenvalues in complex sample covariance matrices were considered, as well as on the Four Moment Theorem for random covariance matrices due
to Tao and Vu (\cite{Tao/Vu:2010b}). Let $p=p(n)$ and $n$ be integers such that $p \geq n$ and $\lim_{n \to \infty} \frac pn = \gamma \in[1, \infty)$.
Let $X$ be a random $p \times n$ matrix with complex entries $X_{ij}$ such that they are identically distributed and independent, have mean zero and variance 1.
Assume moreover that the following condition is fulfilled:

\noindent
{\bf Moment-condition $(M)$}:
$X$ satisfies moment-condition $(M)$, if there exists $C_0 \geq 2$ and $C>0$ such that $\sup_{ij} \E[|X_{ij}|^{C_0}] \leq C$. 

Then $W:= W_{p,n} := \frac 1n X^* X$ is called {\it covariance matrix}. Hence it has at most $p$ non zero eigenvalues, which are real and nonnegative, denoted
by $0 \leq \lambda_1(W) \leq \cdots \leq \lambda_p(W)$. We abbreviate $\lambda_i(W)$ as $\lambda_i$.
If the entries are Gaussian random variables, $W_{p,n}$ is called Laguerre Unitary Ensemble (LUE) or complex 
Gaussian Wishart ensemble. LUE matrices will be denoted by $W'=W_{p,n}'$.
In this case, the eigenvalues form a determinantal point process with admissible, good kernel given in terms of Laguerre polynomials; see
for example \cite{Su:2006}. With Theorem \ref{mdpDDP} we obtain a MDP for the counting function $N_{I_n}(W_{p,n}')$ for intervals $I_n$ with $\V(N_{I_n}(W_{p,n}')) \to \infty$
for $n \to \infty$. The classical Marchenko-Pastur theorem states that as $n \to \infty$ such that $\frac pn \to \gamma \geq 1$, almost surely 
$\frac 1n \sum_{i=1}^n \delta_{\lambda_i} \to \mu_{\gamma}$ in distribution, where $\mu_{\gamma}$ is the Marchenko-Pastur law with density 
$d \mu_{\gamma}(x) = \frac{1}{2 \pi x} \sqrt{(x- \alpha)(\beta-x)} 1_{[\alpha, \beta]}(x) \, dx$, where $\alpha = (\sqrt{\gamma}-1)^2$ and $\beta=(\sqrt{\gamma}+1)^2$.

Let 
$$
\alpha_{p,n} := \biggl( \sqrt{\frac pn} -1 \biggr)^2, \,\, \beta_{p,n} := \biggl( \sqrt{\frac pn} +1 \biggr)^2
$$
and $I_n := [t_n, \infty)$ with $t_n \leq \beta_{n,p} - \delta$ for some $\delta >0$. Then with \cite[Lemma 3]{Su:2006}, the variance of the number
of eigenvalues of $W_{p,n}'$ in $I_n$ satisfies
\begin{equation} \label{suvar}
\V(N_{I_n}(W_{p,n}')) = \frac{1}{2 \pi^2} \log n (1 +o(1)).
\end{equation}
Moreover with \cite[Lemma 1]{Su:2006}, the expected number of eigenvalues of $W_{p,n}'$ in $I_n=[t_n, \infty)$ 
with $t_n \to t \in (\alpha, \beta)$ satisfies
\begin{equation} \label{suvar2}
\E[N_{I_n}(W_{p,n}')] = n \int_{t_n}^{\beta_{p,n}} \mu_{p,n}(x) \, dx \, (1 +o(1)),
\end{equation}
where $\mu_{p,n}(x) := \frac{1}{2\pi x} \sqrt{(x- \alpha_{p,n})(\beta_{p,n}-x)} 1_{[\alpha_{p,n}, \beta_{p,n}]}(x)$. The proof of Theorem \ref{mdpDDP}
leads to an MDP (with numerics) for
$$
\biggl( \frac{N_{I_n}(W_{p,n}') -  n \int_{t_n}^{\beta_{p,n}} \mu_{p,n}(x) \, dx}{a_n\sqrt{
 \frac{1}{2 \pi^2} \log n}} \biggr)_n
$$
for every $I_n=[t_n, \infty)$ with $t_n \to t \in (\alpha, \beta)$ and all $1 \ll a_n \ll \sqrt{\log n}$. Let
$$
G(t) := \int_{\alpha_{p,n}}^t \mu_{p,n}(x) \, dx \,\, \text{for} \,\,  \alpha_{p,n} \leq t \leq \beta_{p,n}.
$$
Arguing as in Section 2, the following MDP can be achieved:

\begin{theorem}
Consider the LUE matrix $W_{p,n}'$. Let $t:=t_{n,i} = G^{-1}(i/n)$ with
$i=i(n)$ such that $\frac in \to a \in (0,1)$ as $n \to \infty$. Then as $\frac pn \to \gamma \geq 1$
the sequence $(X_n)_n$ with
\begin{equation} \label{baldfertig}
X_n := \frac{\sqrt{2} \pi \mu_{p,n}(t) (\lambda_i -t)}{a_n \frac{\sqrt{ \log n}}{n}}
\end{equation}
satisfies an MDP for any $1 \ll a_n \ll \sqrt{\log n}$ with rate $x^2/2$.
\end{theorem}

\begin{proof}
Along the lines of the proof of Theorem \ref{MDP-GUE-Gust}, the main step is to calculate $\E \bigl[ N_{I_n}(W_{p,n}') \bigr]$ for
$I_n = [t + \xi b_n, \infty)$ with $b_n = a_n \frac{\sqrt{\log n}}{n} \frac{1}{\pi \sqrt{2} \mu_{p,n}(t)}$. Using the Taylor expansion
for $\int_{t + \xi b_n}^{\cdot} \mu_{p,n}(x) \, dx$ we obtain
$$
\E \bigl[ N_{I_n}(W_{p,n}') \bigr] = n \int_{t +\xi b_n}^{\beta_{p,n}} \mu_{p,n}(x) \, dx \, (1 +o(1)) = n - i - \xi a_n \sqrt{ \log n} \frac{1}{\sqrt{2} \pi} + o(1).
$$
The statement follows step by step along the proof of Theorem \ref{MDP-GUE-Gust}.
\end{proof}

As was done for Wigner matrices, one can extend the last Theorem to general covariance matrices $W_{p,n}$ whose entries satisfy the moment-condition $(M)$
and match the corresponding entries of LUE up to order 4. Namely, Tao and Vu extended their Four Moment Theorem to the case of covariance matrices
in \cite[Theorem 6]{Tao/Vu:2010b}. We apply it in the same way as for Wigner matrices.
Finally we end up with the following universality result:

\begin{theorem}
Let $W_{p,n}$ be a covariance matrix whose entries satisfy moment-condition $(M)$ and match the corresponding entries of LUE up to order 4.
Then, for any $I_n = [t_n, \infty)$ with $t_n \to t \in (\alpha, \beta)$  and all $1 \ll a_n \ll \sqrt{\log n}$
the sequence $(\hat{Z}_n)_n$ with
$$
\hat{Z}_n = \frac{N_{I_n}(W_{p,n}) -  n \int_{t_n}^{\beta_{p,n}} \mu_{p,n}(x) \, dx}{a_n\sqrt{
 \frac{1}{2 \pi^2} \log n}}
$$
satisfies a MDP (with numerics) with speed $a_n^2$ and rate function $I(x)=\frac{x^2}{2}$.
\end{theorem}

\begin{remark}
Since a version of the Erd\"os-Yau-Yin rigidity theorem for covariance matrices is not yet proved, the MDP
``without numerics'' for the eigenvalue counting function for non-Gaussian covariance matrices is not stated.
\end{remark}

\begin{proof}
For every $\xi \in \R$ we obtain with $X_n$ defined in \eqref{baldfertig} (where $\lambda_i = \lambda_i(W_{p,n})$) that
$$
P_n \bigl( \hat{Z}_n \leq \xi \bigr) = P_n \bigl( X_n \leq \xi_n \bigr)
$$
with
$$
\xi_n = \sqrt{2} \pi \mu_{p,n}(G^{-1}(i_n/n)) \frac{t_n - G^{-1}(i_n/n)}{a_n \frac{\sqrt{\log n}}{n}}
$$
and $i_n := n \int_{\alpha_{p,n}}^{t_n} \mu_{p,n}(x) \, dx - \xi a_n \sqrt{ \frac{1}{2 \pi^2} \log n}$.
Now $i_n/n \to \mu_{\gamma}(-\infty,t]) \in (0,1)$, since $t \in (\alpha, \beta)$. 
Moreover Taylor expansion leads to
$$
G^{-1}(i_n/n)= t_n - \xi \frac 1n a_n \sqrt{ \log n} \frac{1}{\sqrt{2} \pi \mu_{p,n}(t_n)} + o(1).
$$
Hence 
$$
\sqrt{2} \pi \mu_{p,n}(t_n) \frac{t_n - G^{-1}(i_n/n)}{a_n \frac{\sqrt{\log n}}{n}} = \xi +o(1)
$$
and we established the result.
\end{proof}

Real covariance matrices can be considered as well. The first step would be to establish our conclusions for the LOE, the Laguerre
Orthogonal Ensemble. This can be done applying {\it interlacing formulas}, that relates the eigenvalues of LUE and LOE matrices.
Forrester and Rains proved in \cite{Forrester/Rains:2001} the following relation:
$
{\rm LUE}_{p,n} = {\rm even} \bigl( {\rm LOE}_{p,n} \cup {\rm LOE}_{p+1,n+1} \bigr).
$
Now we can conclude to similar MDPs for counting functions of eigenvalues in LOE with respect to intervals in the bulk as well as
for individual eigenvalues in LOE in the bulk. On the basis of the Four Moment Theorem in the real case, the conclusions can be extended
to non-Gaussian real covariance matrices. We omit the details.


\providecommand{\MRhref}[2]{%
  \href{http://www.ams.org/mathscinet-getitem?mr=#1}{#2}
}
\providecommand{\href}[2]{#2}

\newcommand{\SortNoop}[1]{}\def\cprime{$'$} \def\cprime{$'$}
  \def\polhk#1{\setbox0=\hbox{#1}{\ooalign{\hidewidth
  \lower1.5ex\hbox{`}\hidewidth\crcr\unhbox0}}}
\providecommand{\bysame}{\leavevmode\hbox to3em{\hrulefill}\thinspace}
\providecommand{\MR}{\relax\ifhmode\unskip\space\fi MR }
\providecommand{\MRhref}[2]{%
  \href{http://www.ams.org/mathscinet-getitem?mr=#1}{#2}
}
\providecommand{\href}[2]{#2}

\end{document}